\documentclass[11pt,reqno]{amsart}

\newcommand{\version}{June 28, 2026}

\usepackage{amsmath,amssymb,amsthm}
\usepackage[margin=1.1in]{geometry}
\usepackage{hyperref}
\usepackage{graphicx}

\theoremstyle{plain}
\newtheorem{theorem}{Theorem}[section]
\newtheorem{proposition}[theorem]{Proposition}
\newtheorem{lemma}[theorem]{Lemma}

\theoremstyle{definition}

\theoremstyle{remark}
\newtheorem{remark}[theorem]{Remark}

\newcommand{\R}{\mathbb{R}}
\newcommand{\Sph}{\mathbb{S}}
\newcommand{\Prob}{\mathcal{P}}
\newcommand{\gN}{\mathsf{g}_N}
\newcommand{\energy}{\mathcal{I}}
\newcommand{\Cset}{\mathcal{C}}
\newcommand{\half}{\Sigma}
\newcommand{\hyper}{L}
\newcommand{\bulk}{\Omega}
\newcommand{\muhat}{\widehat{\mu}}
\newcommand{\mustar}{\mu_\star}

\newcommand{\ac}{a_{\rm c}}
\newcommand{\supp}{\operatorname{supp}}
\newcommand{\Ham}{\mathcal{H}}

\newcommand{\nuW}{\nu_{\scriptscriptstyle W}}
\newcommand{\muW}{\mu_{{\scriptscriptstyle W}\kern-2.2pt,\kern0.2pt a}}
\newcommand{\RW}{R_{\scriptscriptstyle W}}
\newcommand{\p}{\partial}

\begin{document}
	
    \title[Minimizers for Coulomb gases constrained to a halfspace]{Minimizers for Coulomb gases\\ constrained to a halfspace}
	\author{Rupert L. Frank}
    \address[Rupert L. Frank]{Mathe\-matisches Institut, Ludwig-Maximilians Universit\"at M\"unchen, The\-resien\-str.~39, 80333 M\"unchen, Germany, and Munich Center for Quantum Science and Technology, Schel\-ling\-str.~4, 80799 M\"unchen, Germany, and Mathematics 253-37, Caltech, Pasa\-de\-na, CA 91125, USA}
    \email{r.frank@lmu.de}
    
    \author{Paata Ivanisvili}
    \address[Paata Ivanisvili]{Department of Mathematics, University of California, Irvine, 510C Rowland Hall, Irvine, CA 92697-3875, USA}
    \email{pivanisv@uci.edu}
    
    \author{Clara Torres-Latorre}
    \address[Clara Torres-Latorre]{Instituto de Ciencias Matemáticas, Consejo Superior de Investigaciones Científicas,
        C/ Nicolás Cabrera, 13-15, 28049 Madrid, Spain}
    \email{\tt clara.torres@icmat.es}
    
	\date{\version}
	
	\begin{abstract}
		We consider a family of optimization problems, based on a mean-field description of particles interacting through Coulomb forces in a quadratic trap. In addition, the particles are constrained to lie in a halfspace and we are interested in the way the particle distribution changes as the halfspace varies. In particular, we can prove the existence of a phase transition, thereby settling a recent conjecture by Byun, Forrester, Majumdar and Schehr.
	\end{abstract}

    
	\maketitle

    \section{Introduction and main result}

Recently, there has been a large number of works dedicated to minimizing energy functionals of the form
$$
\frac12 \iint_{\R^N\times\R^N} k(x-y) \,d\mu(x)\,d\mu(y) + \int_{\R^N} V(x) \,d\mu(x)
$$
over all Borel probability measures $\mu$ on $\R^N$. Such optimization problems arise in physics, biology and economics when modeling particles or agents whose interaction among each other and with the environment is described by the functions $k$ and $V$ respectively. They also arise in mathematics in the context of random matrices. Of particular interest in all these applications is the case where $k$ is slowly decaying, corresponding to long range interactions. We refer to \cite{ST97,CDFR14,Fra23,Ser26,BF25} for further background.

Some of these works are concerned with qualitative properties of minimizers \cite{BCLR13,CDM16,CFP17,CS23}, while others aim at finding explicit solutions in some concrete cases \cite{DLM22,DLM23,Fra22,FM25,CMSVW25,Shu25}. The effect of anisotropic kernels $k$ has also been analyzed in detail, for instance in \cite{MRS19,CMMRSV20,CMMRSV21,CS24a,CS24b,MMRSV23,FMMRSV26}.

In this paper we will consider a model suggested in a recent work of Byun, Forrester, Majumdar and Schehr \cite{BFMS26}, with precursors in \cite{ASZ14,DKMSS17}. This model is of interest since it features three distinct `phases' as a certain parameter varies. Our result will give qualitative information about the intermediate, nontrivial phase and will determine the precise value of the phase transition, thereby proving \cite[Conjecture 2.8]{BFMS26}.

Let us be more specific. Throughout the paper, $N$ is the dimension and $\Prob(\R^N)$ denotes the probability measures on $\R^N$. We study the energy functional
\begin{equation}\label{e.energy}
    \energy[\mu]:=\iint \gN(x-y)\,d\mu(x)\,d\mu(y)+\int|x|^2\,d\mu(x),
	\qquad \mu\in\Prob(\R^N),
\end{equation}
with kernel
\begin{equation*}
    	\gN(x) :=
    	\begin{cases}
    		-\log|x|, & N=2,\\[2pt]
    		\dfrac{1}{N-2}\,|x|^{-(N-2)}, & N\neq2.
    	\end{cases}
\end{equation*}
This normalization of the kernel is chosen so that the infimum of $\energy[\mu]$ over $\mu\in\Prob(\R^N)$ is attained at
$$
d\mustar(x)=\tfrac{1}{|B_1|}\,\chi_{B_1}(x) \,dx \,,
$$
with $B_1$ denoting the unit ball in $\R^N$.

The minimization problem that we will consider depends on a parameter $a\in\R$ and consists in determining
\begin{equation}
    \label{e.Ia}
    I(a):=\inf\bigl\{\energy[\mu]:\ \mu\in\Prob(\R^N),\ \supp\mu\subset \{ x\in\R^N:\ x_N \geq a \} \bigr\} \,,
\end{equation}
where we split coordinates in $\R^N$ as $x=(x',x_N)\in\R^{N-1}\times\R$.

Our main results justify the following picture. As $a$ slides up from $-\infty$, the ball is untouched as long as $a\leq -1$ and then pushed. The mass begins to accumulate on the boundary hyperplane, while a bulk blob persists as long as $a<a_{\rm c}$. At the critical height $\ac$, the bulk is fully consumed and the measure collapses entirely onto the hyperplane as a flat profile. For a pictorial representation of this situation, we refer to Figure \ref{f.numerics}.

    \begin{figure}[h]
    \centering
    \includegraphics[width=0.82\textwidth]{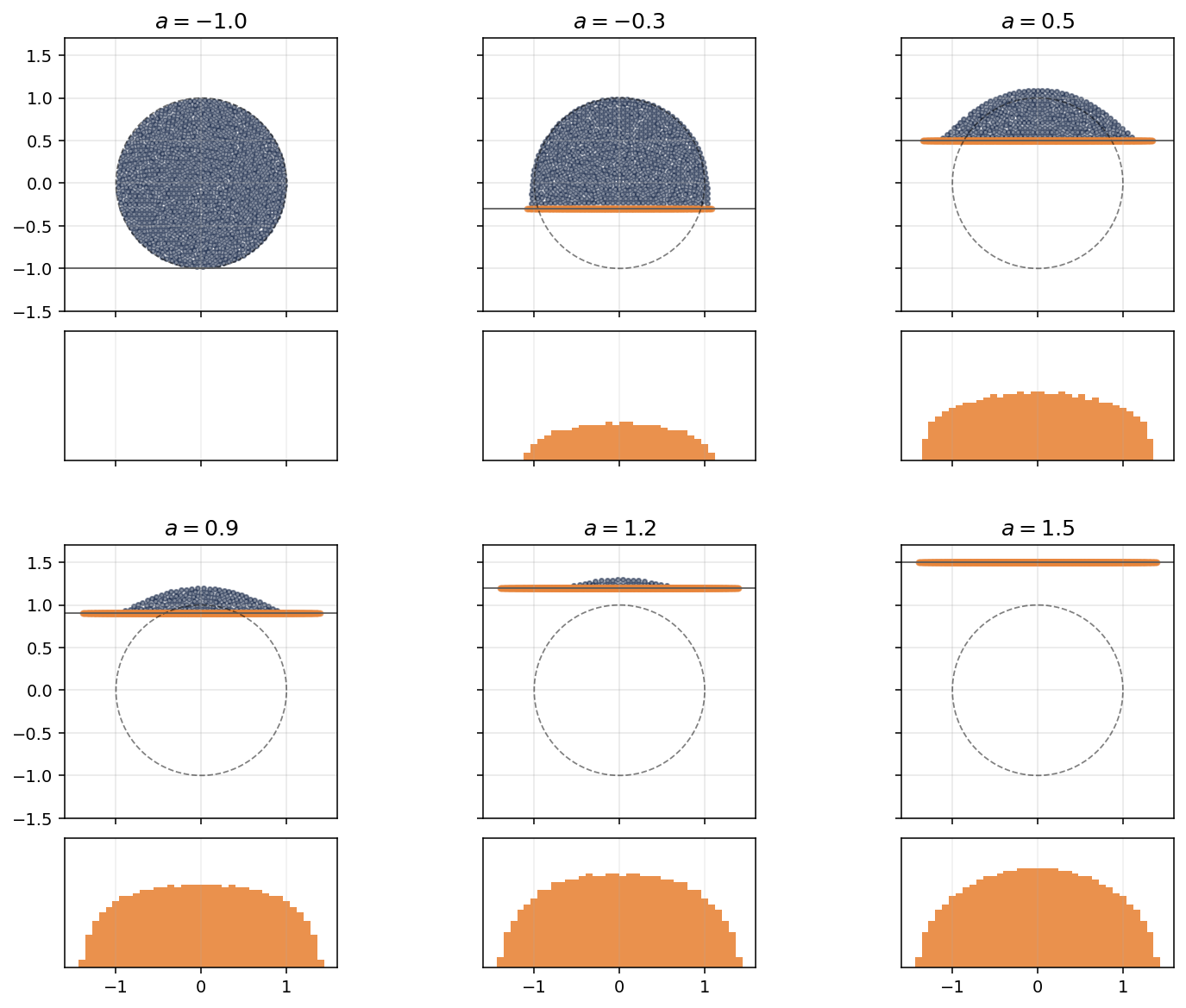}
    \caption{The constrained minimizer $\muhat_a$ for $N=2$ at six constraint levels. In each panel the gray horizontal line is $\{x_2=a\}$; orange particles (within $O(n^{-1/2})$ of it) approximate the singular layer $\muhat_{S,\kern0.2pt a}$, navy particles the bulk $\tfrac1\pi\chi_{\Omega_a}$, and the dashed curve is $\p B_1$. The strip below each panel shows the distribution of the wall mass along $x_1$. The configurations shown are discrete minimizers of a regularized particle energy with $n=1000$ particles, computed using a L-BFGS-B quasi-Newton method with analytic gradients. Details and code available at \url{https://github.com/ClaraTorresLatorre/2D_Coulomb_quadratictrap}.
    }
    \label{f.numerics}
    \end{figure}

A somewhat more quantitative description is the following:
\begin{itemize}
    \item If $a \leq -1$, then $\mustar$ is the unique minimizer for $I(a)$.
    \item If $-1<a<a_{\rm c}$, then there is a unique minimizer for $I(a)$. This minimizer is the sum of an absolutely continuous part, whose density is a constant times the characteristic function of a nonempty, axially symmetric set, and a nontrivial singular part, which is supported on the hyperplane $\{x_N=a\}$ and is absolutely continuous with respect to $(N-1)$-dimensional Lebesgue measure on $\{x_N=a\}$.
    \item If $a\geq a_{\rm c}$, then $I(a)$ is uniquely minimized by the measure $d\nuW(x')\otimes\delta_a(x_N)$ where
    \begin{equation}
        \label{e.nuw}
            	d\nuW(x') = \frac{2\RW\,\Gamma\!\left(\frac{N+1}{2}\right)}{\pi^{\frac{N+1}{2}}}
    	\sqrt{\left(1 - \frac{|x'|^2}{\RW^2}\right)_+}\,dx',
    	\qquad
    	\RW := \left(\frac{\sqrt{\pi}\,\Gamma\!\left(\frac{N}{2}+1\right)}{\Gamma\!\left(\frac{N+1}{2}\right)}\right)^{\!\frac1N}\!\!.
    \end{equation}
\end{itemize}
Moreover, the value where the transition occurs is
\begin{equation}
    \label{e.ac.RW}
    \ac := N\!\left(\frac{\Gamma\!\left(\frac{N+1}{2}\right)}{\sqrt{\pi}\,\Gamma\!\left(\frac{N}{2}+1\right)}\right)^{\!\frac{N-1}{N}}.
\end{equation}
The latter equality is \cite[Conjecture 2.8]{BFMS26}.

These results, and more, appear in our Theorems \ref{t.structure} and \ref{t.classification} below.

\subsection*{Statement of results}
We introduce the notations
\begin{equation*}
    	\half_a:=\{x\in\R^N:\ x_N\ge a\} \,,
    	\qquad \hyper_a:=\p\half_a=\{x_N=a\} \,,
\end{equation*}
as well as
\begin{equation*}
    	\Cset_a:=\bigl\{\mu\in\Prob(\R^N):\ \mu(\half_a)=1,\ \energy[\mu]<\infty\bigr\} \,.
\end{equation*}
Since $\energy$ is strictly convex, lower semicontinuous and coercive, for any $a\in\R$ the infimum in \eqref{e.Ia} is attained at a unique minimizer (see, e.g., \cite{Ser26}):
    \begin{equation*}
    	\muhat_a:=\operatorname*{arg\,min}_{\mu\in\Cset_a}\energy[\mu] \,,
    	\qquad I(a)=\energy[\muhat_a] \,.
    \end{equation*}
Also, we let
$$
\bulk_a:=\operatorname{int}(\supp\muhat_a) \,.
$$
By uniqueness of $\muhat_a$, the set $\bulk_a$ is axially symmetric with respect to the $e_N$-axis.
    
Since $\mustar$ is a minimizer of the unconstrained minimization problem, it is clear that $\muhat_a = \mustar$ for $a\leq -1$, so we will focus on the case $a>-1$.

Many of the properties of $\muhat_a$ will follow from its characterization via an obstacle problem. This is a generalization of \cite[Theorem 1]{ASZ14}. For the potential of a measure $\mu$ we use the notation
\begin{equation*}
    H^\mu(x):=\int \gN(x-y)\,d\mu(y) \,.
\end{equation*}
Since $-\Delta\gN=|\Sph^{N-1}|\,\delta_0$ in $\mathcal D'(\R^N)$, we have
$-\Delta H^\mu=|\Sph^{N-1}|\,\mu$.
    
	\begin{theorem}\label{t.structure}
		The potential $H^{\muhat_a}$ solves the obstacle problem
		\begin{equation*}
			\min\{-\Delta H^{\muhat_a},\ H^{\muhat_a}-\psi_a\}=0\ \text{ in }\R^N,
			\qquad
			\psi_a(x)=
			\begin{cases}
				\tfrac12\bigl(c_a-|x|^2\bigr), & x\in\half_a,\\[2pt]
				-\infty, & x\notin\half_a,
			\end{cases}
		\end{equation*}
		for some constant $c_a$. Moreover:
		\begin{enumerate}
            \item[(a)] The potential satisfies $H^{\muhat_a} = \frac{1}{2}(c_a-|x|^2)$ on $\supp \muhat_a$.
			\item[(b)] The minimizer $\muhat_a$ has compact support, and it decomposes as
			\begin{equation*}
				d\muhat_a=\tfrac{1}{|B_1|}\chi_{\bulk_a} \,dx + d\muhat_{S,\kern0.2pt a},
				\qquad
				d\muhat_{S,\kern0.2pt a}=g_a\,d\Ham^{N-1}\!\restriction_{\hyper_a},
			\end{equation*}
            with $g_a \in L^\infty(\hyper_a)$.\vspace{1mm}
            \item[(c)] For $a_1 < a_2$, $\Omega_{a_1}\cap\{x_N > a_2\} \subseteq \Omega_{a_2}$; in particular, for $a < 1$, $\Omega_a\neq\emptyset$.\vspace{1mm}
			\item[(d)] For $a \in (-1,1)$, $B_1\cap \hyper_a \subseteq \supp\muhat_{S,\kern0.2pt a}$; in particular, $\muhat_{S,\kern0.2pt a} \neq 0$.
		\end{enumerate}
	\end{theorem}
    
Theorem \ref{t.structure} does not address the question of whether the set $\Omega_a$ is empty or not and whether $\muhat_{S,\kern0.2pt a}$ is trivial or not when $a \geq 1$. The following result, which answers both questions, is our main result. We recall that $a_{\rm c}$ is defined in \eqref{e.ac.RW}. For a concise statement, we set
    \begin{equation*}
    	\lambda(a):=\muhat_a(\bulk_a)=\frac{|\bulk_a|}{|B_1|},
    \end{equation*}
    where the second equality comes from Theorem~\ref{t.structure}(b). 

	\begin{theorem}\label{t.classification}
        The minimizer $\muhat_a$ behaves as follows:
		\begin{enumerate}
			\item[(a)] \emph{(Unconstrained.)} If $a\le-1$, then $d\muhat_a=\tfrac{1}{|B_1|}\chi_{B_1}\,dx$;
			in particular $\lambda(a)=1$ and $\muhat_{S,\kern0.2pt a}=0$.\vspace{1mm}
			\item[(b)] \emph{(Coexistence.)} If $-1<a<\ac$, then $\lambda(a)\in(0,1)$ and
			$\muhat_{S,\kern0.2pt a}\ne0$.\vspace{1mm}
			\item[(c)] \emph{(Purely singular.)} If $a\ge\ac$, then $\lambda(a)=0$ and
			$\muhat_a=\muhat_{S,\kern0.2pt a}$ is supported on $\hyper_a$.
		\end{enumerate}
        Moreover, in the purely singular case, we have
        $$d\muhat_a(x) = d\nuW(x')\otimes\delta_a(x_N),$$
        with $d\nuW$ from \eqref{e.nuw}.
	\end{theorem}

In \cite[Theorem 2.7]{BFMS26} it is shown that $\lambda(a)>0$ for $a<a_{\rm c}$ and in \cite[Conjecture 2.8]{BFMS26} it is suggested that $\lambda(a)=0$ for $a\geq a_{\rm c}$. Theorem \ref{t.classification} answers affirmatively this conjecture and, in addition, settles the qualitative behavior of $\muhat_a$ for $a \in (-1,\ac)$. For previous results on this conjecture, we refer to Remark \ref{r.previous} below.

    \subsection*{Acknowledgements}
    
    Much of this work was carried out during the visits of P.I. and C.T.-L. to Munich. They wish to thank R.L.F. and the Department of Mathematics at the Ludwig-Maximilians-Universität München for their support and hospitality.

    R.L.F. acknowledges partial support through US National Science Foundation grant DMS-1954995, as well as through the German Research Foundation through EXC-2111-390814868, TRR 352-Project-ID 470903074 and FR 2664/3-1. P.~I.~acknowledges partial support from the US NSF CAREER grant DMS-2152401, US NSF grant DMS-2554183, a Simons Fellowship, and a Humboldt Research Fellowship for Experienced Researchers.
    C.T.-L. has received funding from the European Research Council (ERC) under the Grant Agreement No. 862342, from AEI project PID2024-156429NB-I00 (Spain), and from the Grant CEX2023-001347-S funded by MICIU/AEI/10.13039/501100011033 (Spain).   

    \subsection*{Use of generative AI}

    The authors acknowledge the use of AI tools during the exploratory stage of this project. All mathematical arguments and proofs in the final manuscript were checked and written by the authors.

\section{Connection to the obstacle problem}\label{s.obstacle}

First, we recall that the Euler--Lagrange equations (also known as Frostman relations in this context) uniquely determine the minimizer. Below, ``q.e.'' (quasi-everywhere) means outside a set of zero capacity.

    \begin{lemma}\label{l.char}
	Let $\mu\in\Cset_a$. Then $\mu=\muhat_a$ if and only if there is a constant $c_a\in\R$ with
	\begin{equation}\label{e.frostman}
		2H^{\mu}+|x|^2\ \ge\ c_a\quad\text{q.e.\ on }\half_a,
		\qquad
		2H^{\mu}+|x|^2\ =\ c_a\quad\text{q.e.\ on }\supp\mu.
	\end{equation}
    \end{lemma}
    
    \begin{proof}
        The proof is standard \cite{La72}. For instance, one can follow that of \protect{\cite[Lemmas 2.1, 2.2]{ASZ14}}, replacing the logarithmic kernel by the Coulomb potential $\gN$.
    \end{proof}

    The constant $c_a$ is sometimes referred to as the Frostman constant. When $a \leq -1$, $\muhat_a = \mustar$, and we will write $c_\star$ instead of $c_a$.

    Our next step is parallel to \cite[Lemma 2.3]{ASZ14}.
    \begin{lemma}\label{l.obstacle}
        Let $N \geq 2$. The function $H^{\muhat_a}$ is the unique element of
        $$\mathcal{K} := \left\{v \in H^1_{\mathrm{loc}} \ : \ v - H^{\muhat_a} \text{ has bounded support and } 2v + |x|^2 \geq c_a \quad \text{q.e.\ on }\half_a\right\}$$
        that satisfies
        $$\int_{\R^N}\nabla H^{\muhat_a}(y)\cdot\nabla(v-H^{\muhat_a})(y)\,dy \geq 0, \quad \forall v \in \mathcal{K}.$$
    \end{lemma}

    \begin{proof}
        If $N = 2$, this is \cite[Lemma 2.3]{ASZ14}. If $N \geq 3$, we show $\nabla H^{\muhat_a}\in L^2(\R^N)$; the rest then follows as in \cite[Lemma 2.3]{ASZ14}.
        
        Write $\mu:=\muhat_a$ and
        $\mathcal F f(\xi)=\int f(x)\,e^{-2\pi i x\cdot\xi}\,dx$. Then, using $\mathcal F\gN(\xi)=|\Sph^{N-1}|(2\pi|\xi|)^{-2} \in L^1_{\mathrm{loc}}$ and $\mathcal F H^{\mu}=(\mathcal F\gN)\,\mathcal F\mu$, Plancherel gives
        \begin{align*}
        \int_{\R^N}|\nabla H^{\mu}|^2
        &=\int_{\R^N}(2\pi|\xi|)^2|\mathcal F\gN|^2|\mathcal F\mu|^2\,d\xi
        =|\Sph^{N-1}|\int_{\R^N}\mathcal F\gN\,|\mathcal F\mu|^2\,d\xi\\
        &=|\Sph^{N-1}|\!\iint\gN(x-y)\,d\mu\,d\mu
        \le|\Sph^{N-1}|\,\energy[\mu]<\infty,
        \end{align*}
        where in the last step we used $\mu \in \Cset_a$.
    \end{proof}

    Then, we need the following technical result.
    \begin{lemma}\label{l.compact}
    Let $N\ge2$ and $a\in\R$. Then,
    \begin{enumerate}
        \item[(i)] The minimizer $\muhat_a$ has compact support.
        \item[(ii)] The map $a\mapsto c_a$ is nondecreasing. In particular, $c_a\ge c_\star$ for every $a\in\R$.
    \end{enumerate}
\end{lemma}

\begin{proof}
    \emph{(i)} We look at the Frostman relations \eqref{eq:frostman}. In dimension $N \geq 3$, $H^{\muhat_a} \geq 0$, so $|x|^2 \leq c_a$ in $\supp\muhat_a$ and hence $\supp\muhat_a$ is bounded.

    In dimension $N = 2$, first note that $|x-y| \leq (1+|x|)(1+|y|)$, and since
    $\log(1+t) \leq t \leq 1 + \tfrac{t^2}{4}$,
    \[
    -\log|x-y| \geq -\log(1+|x|) - \log(1+|y|) \geq -2 - \frac{|x|^2 + |y|^2}{4}.
    \]
    Integrating against $d\muhat_a(x)\,d\muhat_a(y)$ and using the fact that $\muhat_a$ is a
    probability measure,
    \[
    \energy[\muhat_a]
    = \iint -\log|x-y|\,d\muhat_a\,d\muhat_a + \int |x|^2\,d\muhat_a
    \geq -2 + \frac{1}{2}\int |x|^2\,d\muhat_a,
    \]
    and then $\int |x|^2\,d\muhat_a < \infty$. In particular
    $C := \int \log(1+|y|)\,d\muhat_a(y) \leq \int |y|\,d\muhat_a(y) < \infty$, and
    then
    $$H^{\muhat_a}(x) = \int -\log|x-y|\,d\muhat_a(y) \geq -\log(1+|x|) - \int \log(1+|y|)\,d\muhat_a(y) = -\log(1+|x|)-C.$$
    On $\supp\muhat_a$ we then have
    $c_a = |x|^2 + 2H^{\muhat_a}(x) \geq |x|^2 - 2\log(1+|x|) - 2C$, so $\supp\muhat_a$
    is bounded. Thus we have shown that $\muhat_a$ has compact support in any dimension $N\geq 2$.

    \medskip

    \emph{(ii)} Let $a<b$; we must show $c_a\le c_b$. Assume for contradiction that
    $$\varepsilon:=\tfrac12(c_a-c_b)>0.$$
    Since $a<b$, $\half_b\subseteq\half_a$, so $\supp\muhat_b\subseteq\half_b\subseteq\half_a$. On $\supp\muhat_b$, Lemma \ref{l.char} gives $2H^{\muhat_b}+|x|^2=c_b$, while $2H^{\muhat_a}+|x|^2\ge c_a$ q.e.\ on $\half_a$; subtracting,
    \begin{equation}\label{e.mono.ineq}
        H^{\muhat_b}\le H^{\muhat_a}-\varepsilon\qquad\muhat_b\text{-a.e.}
    \end{equation}
    For $N=2$, \eqref{e.mono.ineq} leads to a contradiction exactly as in case~II of the proof of \cite[Theorem 3.2]{ST97}. We explain how to modify the proof to treat $N \geq 3$.

    The function $\min\{H^{\muhat_b},H^{\muhat_a}-\varepsilon\}$ is superharmonic. By the Riesz decomposition theorem \cite[Theorem 1.22]{La72}, there is a nonnegative Borel measure $\lambda$ on $\R^N$ such that for every $r>0$ there is a function $h_r$ harmonic in $B_r$ with
    \begin{equation}\label{e.mono.riesz}
        \min\{H^{\muhat_b},H^{\muhat_a}-\varepsilon\}=h_r+H^{\lambda\!\,\restriction_{\!B_r}}\qquad\text{in }B_r.
    \end{equation}
    (This is where the assumption $N\geq 3$ enters. Also note that the measure $\lambda$ is independent of $r$, which is what the proof gives, even if this is not stated in the theorem.)
    
    By part (i), $\muhat_a$ and $\muhat_b$ have compact support, and then $H^{\muhat_a},H^{\muhat_b}\to0$ at infinity. Choose $R>0$ large enough so that $\supp\hat\mu_a$ is contained in $\overline{B_R}$ and that $\min\{ H^{\hat\mu_b},H^{\hat\mu_a}- \varepsilon\} = H^{\hat\mu_a}-\varepsilon$ in $\R^N\setminus \overline{B_R}$. Thus, for all $r> R$, we have
    $$
    H^{\hat\mu_a}-\varepsilon = h_r + H^{\lambda|_{B_r}}
    \qquad\text{in}\ B_r\setminus \overline{B_R} \,.
    $$
    Since the left side and $h_r$ are harmonic in $B_r \setminus\overline{B_R}$, so is $H^{\lambda|_{B_r}}$. This implies that $\lambda|_{B_r}$ vanishes in $B_r \setminus\overline{B_R}$. Since $r> R$ is arbitrary, we conclude that $\supp\lambda \subset\overline{B_R}$.
    
    Thus, $H^{\lambda|_{B_r}} = H^\lambda$ for all $r> R$. It follows from \eqref{e.mono.riesz} that there is a harmonic function $h$ on $\R^N$ such that
    $$
    \min\{ H^{\hat\mu_b},H^{\hat\mu_a}- \varepsilon\} = h + H^{\lambda}
    \qquad\text{in}\ \R^N \,.
    $$
    We let $|x|\to\infty$ in this equation. Since the left side tends to $-\varepsilon$ and $H^\lambda$ tends to zero, we conclude that $h(x)\to-\varepsilon$ as $|x|\to\infty$. Therefore, by Liouville's theorem, $h\equiv -\varepsilon$. Thus, we obtain
    \begin{equation}
    	\label{e.mono.hlambda}
    	\min\{ H^{\hat\mu_b},H^{\hat\mu_a}- \varepsilon\} = H^{\lambda} - \varepsilon
    	\qquad\text{in}\ \R^N \,.
    \end{equation}  
    In particular $H^{\muhat_a}=H^\lambda$ in $\R^N\setminus\overline{B_R}$, and comparing the coefficients of $|x|^{-(N-2)}$ as $|x|\to\infty$ we obtain
    \begin{equation}\label{e.mono.mass}
        \lambda(\R^N)=\muhat_a(\R^N)=1.
    \end{equation}
    Writing $\|\rho\|^2:=\iint\gN(x-y)\,d\rho(x)\,d\rho(y)$ for the (suitably normalized) Coulomb norm of $\rho$, we first verify that $\|\lambda\|<\infty$. Indeed, using \eqref{e.mono.hlambda}, \eqref{e.mono.mass} and \eqref{e.mono.ineq}, we find
    \begin{align*}
        \|\lambda\|^2 & = \int H^{\lambda}\,d\lambda \\
        & = \int \min\{ H^{\hat\mu_b} + \varepsilon,H^{\hat\mu_a} \} \,d\lambda \\
        & \leq \int H^{\hat\mu_b} \,d\lambda + \varepsilon \\
        & = \int H^{\lambda} \,d\muhat_b + \varepsilon \\
        & = \int (H^{\muhat_b}+\varepsilon) \,d\muhat_b + \varepsilon \\
        & = \|\muhat_b\|^2 + 2\varepsilon <\infty \,.
    \end{align*}
    Here in the next to last equality we used
    \begin{equation}
        \label{e.mono.hlambdaineq}
        H^{\muhat_b} = \min\{H^{\muhat_b},H^{\muhat_a}-\varepsilon\} \quad \muhat_b\text{-a.e.},
    \end{equation}
    from \eqref{e.mono.ineq} and \eqref{e.mono.hlambda}.
    
    Next, using \eqref{e.mono.hlambda} and \eqref{e.mono.mass}, we compute
    \begin{align*}
        \|\muhat_b-\lambda\|^2
        &=\int\bigl(H^{\muhat_b}-H^\lambda\bigr)\,d(\muhat_b-\lambda)\\
        &=\int\bigl(H^{\muhat_b}-\min\{H^{\muhat_b},H^{\muhat_a}-\varepsilon\}-\varepsilon\bigr)\,d(\muhat_b-\lambda)\\
        &=\int\bigl(H^{\muhat_b}-\min\{H^{\muhat_b},H^{\muhat_a}-\varepsilon\}\bigr)\,d(\muhat_b-\lambda)\\
        &=-\int\bigl(H^{\muhat_b}-\min\{H^{\muhat_b},H^{\muhat_a}-\varepsilon\}\bigr)\,d\lambda\ \le\ 0,
    \end{align*}
    where in the last equality we used again \eqref{e.mono.hlambdaineq}. 
    
    The inequality $\| \hat\mu_b - \lambda \|^2\leq 0$ implies that $\lambda = \hat\mu_b$. Reinserting this into \eqref{e.mono.hlambda} gives
    $$
    \min\{ H^{\hat\mu_b},H^{\hat\mu_a}- \varepsilon\} = H^{\hat\mu_b} - \varepsilon
    \qquad\text{in}\ \R^N,
    $$
    which implies
    $$
    H^{\hat\mu_a} \leq H^{\hat\mu_b} \,,
    $$
    contradicting $\varepsilon > 0$. Thus $a\mapsto c_a$ is nondecreasing, and in particular, since $c_a = c_\star$ for $a \leq -1$, $c_a \geq c_\star$ for all $a \in \R$.
\end{proof}

    Finally, we prove our first main result. The proof of items (a)--(b) follows \cite[Theorem 1]{ASZ14} closely; the only change is that the obstacle equals $-\infty$ on $\{x_N < a\}$. The proof of items (c)--(d) is a simplification of the corresponding argument in \cite{ASZ14}, that now also needs the bound $c_a \geq c_\star$ from Lemma \ref{l.compact}(ii) because the maximum principle on the unbounded components requires control of the function at infinity for $N \geq 3$.
	\begin{proof}[Proof of Theorem~\ref{t.structure}]
        For $N = 1$ the proof is a computation that follows from $\ac = 1$ and
        \[
        d\muhat_a =
        \begin{cases}
            \tfrac{1}{2}\chi_{(-1,1)}\,dx, & a \leq -1,\\[3pt]
            \tfrac{1}{2}\chi_{(a,1)}\,dx + \tfrac{a+1}{2}\,\delta_a, & -1<a<1,\\[3pt]
            \delta_a, & a \geq 1,
        \end{cases}
        \]
        which can be checked by elementary means (cf. \cite{DKMSS17}).
    
        Now we focus on $N \geq 2$. First, Lemma \ref{l.obstacle} implies that $H^{\muhat_a}$ is a weak solution of the obstacle problem with obstacle $\psi_a$. By classical regularity results (see for instance \cite[Chapter 5]{FR22}), $H^{\muhat_a} \in C^{1,1}_{\mathrm{loc}}(\R^N\setminus\hyper_a)$. Moreover, $H^{\muhat_a}$ is continuous across the hyperplane $\hyper_a$ \cite{FM82}. Therefore,
        \begin{equation}
            \label{eq:frostman}
            2H^{\muhat_a} + |x|^2 = c_a \text{ in } \supp\muhat_a \quad \text{and} \quad 2H^{\muhat_a}+|x|^2 \geq c_a \text{ in } \half_a.
        \end{equation}
        
        By the definition of the potential,
        $$d\muhat_a = -\frac{1}{|\Sph^{N-1}|}\Delta H^{\muhat_a} = \frac{1}{|B_1|}\chi_{\Omega_a}\,dx \quad \text{in} \ \R^N\setminus\hyper_a.$$
        Indeed, this is clear in the open sets $\Omega_a$ and $\R^N\setminus\Omega_a$, but in fact it holds a.e.~in $\R^N\setminus L_a$. Indeed, the function $2H^{\muhat_a} + |x|^2$ is twice weakly differentiable, so its Laplacian vanishes a.e.~on the set $\{2H^{\muhat_a} + |x|^2 = c_a\}$; see, e.g., \cite{FL18}. We also use the fact that, since $H^{\muhat_a} \in C^{1,1}_{\mathrm{loc}}(\R^N\setminus\hyper_a)$, the distributional Laplacian of $H^{\muhat_a}$, restricted to $\R^N\setminus L_a$, belongs to $L^\infty_{\mathrm{loc}}(\R^N\setminus L_a)$. 

        We write $d\muhat_{S,\kern0.2pt a} := d\muhat_a - \tfrac{1}{|B_1|}\chi_{\Omega_a}\,dx$, which is supported on $\hyper_a$.
              
        Then, we prove that $H^{\muhat_a}$ is locally Lipschitz across $\hyper_a$. To show this, let $x_0 \in \hyper_a$. Since by Lemma \ref{l.compact}(i), $\supp\muhat_a \subset B_R$, we have that for $N=2$,
        $$H^{\muhat_a} = \gN * \muhat_a \geq -\log(|x|+R),$$
        and $H^{\muhat_a} \geq 0$ if $N \geq 3$. In both cases, we can replace the obstacle by
        $$\tilde\psi_a(x)=
			\begin{cases}
				\tfrac12\bigl(c_a-|x|^2\bigr), & x\in\half_a,\\[2pt]
				h(x), & x\in B_1(x_0)\setminus\half_a,
			\end{cases}  
        $$
        where $h$ is the solution to
        $$
        \left\{
        \begin{array}{rclll}
             \Delta h & = & 0 & \text{in} & B_1(x_0)\setminus\half_a\\
             h & = & \tfrac12\bigl(c_a-|x|^2\bigr) & \text{on} & B_1(x_0)\cap\hyper_a\\
             h & = & \!-\log(|x|+R) & \text{on} & \p B_1(x_0)\setminus\half_a,
        \end{array}
        \right.
        $$
        and $H^{\muhat_a}$ is still a solution to the obstacle problem with the obstacle $\tilde\psi_a$ because $H^{\muhat_a} > \tilde\psi_a$ on $\R^N\setminus\half_a$. Then, since $\tilde\psi_a \in C^{0,1}(B_1(x_0))$, $H^{\muhat_a} \in C^{0,1}(B_{1/2}(x_0))$ by \cite[Theorem 2(a)]{Caf98}.

        Now, $H^{\muhat_a}$ is a solution to
        $$-\Delta H^{\muhat_a} = N\chi_{\Omega_a} \quad \text{in} \ \R^N\setminus\hyper_a.$$
        Since $-\Delta H^{\muhat_a}\in L^\infty$ on each open side of $\hyper_a$ and the boundary datum on $\hyper_a$ is Lipschitz, $H^{\muhat_a}$ has one-sided Neumann traces
        $$
        T^{\pm} := \p_{x_N}\bigl(H^{\muhat_a}|_{\half_a}\bigr), \
        \p_{x_N}\bigl(H^{\muhat_a}|_{\R^N\setminus\half_a}\bigr) \in L^2_{\mathrm{loc}}(\hyper_a).
        $$

        We will prove that $T^{\pm}$ are bounded: Let $L$ be the Lipschitz constant of $H^{\muhat_a}$ on a neighborhood of $\hyper_a\cap\supp\muhat_a$. Since $H^{\muhat_a}\in C^{1,1}_{\mathrm{loc}}(\R^N\setminus\hyper_a)$, the gradient $\nabla H^{\muhat_a}$ is continuous on each open side of $\hyper_a$, so the a.e.\ bound $|\nabla H^{\muhat_a}|\le L$ holds everywhere there, and in particular the interior slices $\p_{x_N}H^{\muhat_a}(\cdot,a\pm\varepsilon)$ are bounded by $L$ uniformly in $\varepsilon>0$. Any weak-$*$ limit point of these slices as $\varepsilon\downarrow0$ solves the same Green identity that defines $T^{\pm}$, using that $\nabla H^{\muhat_a}$ and $\Delta H^{\muhat_a}$ are bounded up to $\hyper_a$ on each open side; hence it equals $T^{\pm}$, and the slices converge weakly-$*$ to $T^{\pm}$. By weak lower semicontinuity of the norm, $\|T^{\pm}\|_{L^\infty(\hyper_a)}\le L$.        

        Moreover, integrating by parts: for $\zeta\in C_c^\infty(\R^N)$, using
        $-\Delta H^{\muhat_a}=|\Sph^{N-1}|\,\muhat_a$,
        \begin{align*}
            |\Sph^{N-1}|\!\int_{\R^N}\!\zeta\,d\muhat_a
            &= \int_{\R^N}\!\nabla H^{\muhat_a}\cdot\nabla\zeta\,dx\\
            &= |\Sph^{N-1}|\!\int_{\Omega_a}\!\zeta\,\frac{dx}{|B_1|}
            + \int_{\hyper_a}\!\zeta\Bigl(\p_{x_N}\bigl(H^{\muhat_a}|_{\R^N\setminus\half_a}\bigr) -\p_{x_N}\bigl(H^{\muhat_a}|_{\half_a}\bigr)
             \Bigr)d\mathcal H^{N-1}.
        \end{align*}
        Hence $\muhat_{S,\kern0.2pt a}=g\,\mathcal H^{N-1}\!\restriction_{\hyper_a}$ with
        \[
            g := \frac{1}{|\Sph^{N-1}|}\Bigl(\p_{x_N}\bigl(H^{\muhat_a}|_{\R^N\setminus\half_a}\bigr) - \p_{x_N}\bigl(H^{\muhat_a}|_{\half_a}\bigr)\Bigr)\in L^\infty(\hyper_a),
            \qquad g\geq 0.
        \]

        To prove (c), let $a_1 < a_2$ and define
        $$w := H^{\muhat_{a_1}} - H^{\muhat_{a_2}} - \tfrac12\bigl(c_{a_1} - c_{a_2}\bigr).$$
        First, note that
        $$-\Delta w \geq 0 \text{ in } \R^N\setminus\supp\muhat_{a_2} \quad \text{and} \quad -\Delta w \leq 0 \text{ in } \R^N\setminus \supp\muhat_{a_1}.$$

        Moreover, $H^{\muhat_{a_1}} \geq \tfrac12\bigl(c_{a_1}-|x|^2\bigr)$ and $H^{\muhat_{a_2}} = \tfrac12\bigl(c_{a_2}-|x|^2\bigr)$ in $\supp\muhat_{a_2}$, and from the definition of the potential and the supports of both $\muhat_{a_1}$ and $\muhat_{a_2}$ being compact, $H^{\muhat_{a_1}} - H^{\muhat_{a_2}} \rightarrow 0$ at infinity, which implies $w \rightarrow \tfrac12\bigl(c_{a_2} - c_{a_1}\bigr) \geq 0$ by Lemma \ref{l.compact}(ii). Hence, by the comparison principle, $w \geq 0$.

        From $w \geq 0$, $H^{\muhat_{a_2}} \leq H^{\muhat_{a_1}} - \tfrac12\bigl(c_{a_1}-c_{a_2}\bigr)$, so $H^{\muhat_{a_2}} = \tfrac12(c_a-|x|^2)$ in $\supp\muhat_{a_1}\cap\half_{a_2}$, then $\Omega_{a_1}\cap\{x_N > a_2\} \subseteq \supp\muhat_{a_2}$, and then $\Omega_{a_1} \cap\{x_N > a_2\} \subseteq \Omega_{a_2}$. Finally, if $a \leq -1$, $\Omega_a = B_1$, and if $a \in (-1,1)$, considering the particular case $a_1 = -1, a_2 = a$ gives $\Omega_a \neq \emptyset$.

        Finally, to prove (d), set $w$ as in (c) with $a_1 = -1$, $a_2 = a$, so that $w = H^{\mustar} - H^{\muhat_a} - \tfrac12(c_\star - c_a)$. First note that
        $$\supp\muhat_a\cap\{w = 0\} \subseteq \overline{B_1}\cap\half_a \subseteq \{w = 0\}.$$
        Indeed, if $x \in \supp\muhat_a\cap\{w = 0\}$, 
        $$H^{\muhat_a}(x) = \tfrac12(c_a-|x|^2) \ \Rightarrow\ H^{\mustar}(x) = \tfrac12(c_\star-|x|^2) \ \Rightarrow\ x \in \overline{B_1},$$
        which together with $\supp\muhat_a \subset \half_a$ proves the first inclusion. Furthermore, $\overline{B_1}\cap\half_a \subseteq \supp\muhat_a$ implies $H^{\muhat_a}(x) = \tfrac12(c_a-|x|^2)$ in $\overline{B_1}\cap\half_a$, and hence $w = 0$ in $\overline{B_1}\cap\half_a$.
      
        Now, suppose that there is $x \notin \supp\muhat_a$ with $w(x) = 0$. Then, by the strong maximum principle, $w \equiv 0$ in $U$, the connected component of $\R^N\setminus\supp\muhat_a$ containing $x$. Now, if $U$ is bounded, $U \subset \half_a$, $H^{\muhat_a} = \tfrac12\bigl(c_a-|x|^2\bigr)$ on $\p U$, and $\Delta H^{\muhat_a} = 0$ in $U$ but $H^{\muhat_a} \geq \tfrac12\bigl(c_a-|x|^2\bigr)$, a contradiction with $\tfrac12\bigl(c_a-|x|^2\bigr)$ being strictly superharmonic. On the other hand, if $U$ is unbounded, since $\supp\muhat_a$ is bounded, $U$ contains $\R^N\setminus\half_a$, in particular $B_1\setminus\half_a$, and in the latter set $-\Delta w = -\Delta H^{\mustar} = N$, contradicting that $w = 0$. Hence, $\{w = 0\} \subset \supp\muhat_a$, and in conclusion $\{w = 0\} = \overline{B_1}\cap\half_a$.

        Finally, assume that there exists $x_0 \in B_1\cap\hyper_a$ but $x_0 \notin \supp\muhat_{S,\kern0.2pt a}$, and choose $\rho > 0$ such that $B_\rho(x_0) \subset B_1\setminus\supp\muhat_{S,\kern0.2pt a}$. Then, $\nabla w$ is continuous at $x_0$, and since $w = 0$ on $\overline{B_1}\cap\half_a$, $\nabla w (x_0) = 0$. On the other hand, $w > 0$ and $-\Delta w \geq 0$ in $B_\rho(x_0)\cap\{x_N < a\}$, so by the Hopf lemma $\p_{x_N}w(x_0) < 0$, a contradiction.
	\end{proof}

\section{Proof of \texorpdfstring{\cite[Conjecture 2.8]{BFMS26}}{[BFMS26, Conjecture 2.8]}}\label{s.conjecture}
    For a fixed $a\in\R$, we write $\muW := \nuW \otimes\delta_a$ and $H^{\muW} := \gN * \muW$. The goal of this section is to prove the following:
	\begin{proposition}\label{p.threshold}
		Let $N \geq 3$ and $a\geq a_{\rm c}$. Then $\muW$ is the unique minimizer attaining~$I(a)$.
	\end{proposition}

    \begin{remark}\label{r.previous}
    The conjecture was already known in the cases $N\in\{1,2,4\}$. For $N=1$, Proposition \ref{p.threshold} is folklore and known in the physics literature \cite{DKMSS17}. In $\R^2$, $\muW$ is the semicircle law, shown to be the minimizer for $a\ge\ac=\sqrt2$ in \cite[Proposition 3.1]{ASZ14}. Finally, the conjecture was proved in $\R^4$ in \cite[Appendix A]{BFMS26}. Proposition~\ref{p.threshold} recovers the latter as part of
    an argument valid for all $N\ge3$. Also, our proof can be adapted to treat the cases $N=1,2$. (For $N=2$ one needs to `differentiate' several identities for $|x|^{-N+2}$ at $N=2$. This is carried out in detail in \cite{FM25} in the context of a related, but different question and we refrain from carrying out the details since the final result is already known from \cite{ASZ14}.)
    \end{remark}

    We first recall the Euler--Lagrange relations of the auxiliary $(N-1)$-dimensional problem.
    \begin{lemma}\label{l.hyperplane}
	The generalized Wigner law $\nuW$ is the unique minimizer attaining
	$$
	\inf\{ \energy[\nu\otimes\delta_a] :\ \nu\in \Prob(\R^{N-1}) \} \,,
	$$
	and there is a constant $c'\in\R$ such that
	\begin{equation}
		\label{eq:eln-1}
		\begin{cases}
			2H^{\muW}(x',a) + |x'|^2 = c' & \text{on}\ \{ |x'|\leq \RW\} \,, \\
			2H^{\muW}(x',a) + |x'|^2 \geq c' & \text{on}\ \R^{N-1} .
		\end{cases}
	\end{equation}
\end{lemma}

\begin{proof}
    On $\hyper_a$ the kernel $\gN$ restricts to the Riesz kernel of order $N-2$ on $\R^{N-1}$, so $\nuW$ is the unique minimizer of \cite[Theorem~2]{FM25} with $d = N-1$, $\alpha = 2$ and $\beta = 2-N$\footnote{Note that the energy in \cite{FM25} is a pure pair interaction,
$E_{\alpha,\beta}[\mu]=\tfrac12\iint\bigl(\tfrac{1}{\alpha}|x-y|^{\alpha}-\tfrac{1}{\beta}|x-y|^{\beta}\bigr)\,d\mu(x)\,d\mu(y)$,
but for $\alpha=2$ its quadratic term is proportional to our external field up to the
center-of-mass term $|\!\int x'\,d\nu|^2$. By translation invariance of $E_{\alpha,\beta}$, the latter can be assumed to vanish.}, and \eqref{eq:eln-1} is its Euler--Lagrange relation; see also \cite{CV11}.
\end{proof}

The following lemma is the key ingredient in the proof of the proposition. Since $\nuW$ is rotation invariant, $H^{\muW}(x',x_N)$ depends on $x'$ only through $|x'|$. Define
\[
G(r,t):=2H^{\muW}(x',a+t)+r^2,
\qquad r=|x'|,\ t\ge0.
\]

\begin{lemma}\label{mono}
For all $r,t\ge0$,
\[
\partial_r G(r,t)\ge0.
\]
\end{lemma}

\begin{proof}
	\textit{Step 1.} We show that $\p_rG(r,0)\geq 0$ for $r\geq 0$.
	
	For $r\leq \RW$ we have $G(r,0)= c'$ by Lemma \ref{l.hyperplane}. For $r>\RW$, we express $2H^{\muW}$ in terms of the hypergeometric series
	$$
	F(a,b;c;z) := \sum_{n=0}^\infty \frac{a(a+1)\cdots(a+n-1)}{c(c+1)\cdots(c+n-1)} \,\frac{b(b+1)\cdots(b+n-1)}{1\cdot 2\cdots n} \, z^n .
	$$
	Since $\muW = \nuW\otimes\delta_a$ and $|x-y|=|x'-y'|$ for $x,y\in\hyper_a$,
	$$
	2H^{\muW}(x',a) = \frac{2}{N-2}\cdot\frac{2\RW\,\Gamma(\frac{N+1}{2})}{\pi^{\frac{N+1}{2}}}\int_{|y'|\le\RW} |x'-y'|^{-(N-2)}\sqrt{1-\frac{|y'|^2}{\RW^2}}\,dy'.
	$$
	Rescaling $y'=\RW\eta$ gives
	$$
	2H^{\muW}(x',a) = \frac{2}{N-2}\cdot\frac{2\,\Gamma(\frac{N+1}{2})}{\pi^{\frac{N+1}{2}}}\,\RW^{2}\int_{|\eta|\le 1} \Bigl|\tfrac{x'}{\RW}-\eta\Bigr|^{-(N-2)}\sqrt{1-|\eta|^2}\,d\eta.
	$$
	The integral is computed using \cite[Lemma 9]{FM25} with $\gamma=-(N-2)$ and $d=N-1$, at the point $x'/\RW$ (of modulus $\geq 1$):
	$$
	\int_{|\eta|\le 1} \Bigl|\tfrac{x'}{\RW}-\eta\Bigr|^{-(N-2)}\sqrt{1-|\eta|^2}\,d\eta
	= \frac{\pi^{\frac N2}}{2\,\Gamma(\frac N2+1)}\,\RW^{\,N-2}\,|x'|^{-(N-2)}\,F\!\left(\tfrac{N-2}{2},\tfrac12;\tfrac{N+2}{2};(|x'|/\RW)^{-2}\right).
	$$
	Combining the two previous expressions, the constants simplify using $\RW^{N} = \sqrt\pi\,\Gamma(\frac N2+1)/\Gamma(\frac{N+1}{2})$, and we obtain
	$$
	2H^{\muW}(x',a) = \frac{2}{N-2}\,|x'|^{-(N-2)}\,F\!\left(\tfrac{N-2}{2},\tfrac12;\tfrac{N+2}{2};(|x'|/\RW)^{-2}\right).
	$$
    
    Directly from the definition of the hypergeometric series we find that
	$$
	\frac{d}{dz}F(a,b;c;z) = \frac{ab}{c}\, F(a+1,b+1;c+1;z) .
	$$

    It follows that
	$$\p_r G(r,0) = -2\,r^{-(N-1)}F(\tfrac{N-2}{2},\tfrac12;\tfrac{N+2}{2};z) - \frac{4z}{N-2}\,r^{-(N-1)}\,\tfrac{N-2}{2(N+2)}F(\tfrac{N}{2},\tfrac32;\tfrac{N+4}{2};z) + 2r.$$
	where $z=(r/\RW)^{-2}$. Dividing by $2r$ and using $r^{-N}=\RW^{-N}z^{\frac N2}$,
	\begin{equation}
		\label{eq:derivative}
		\frac{\p_r G(r,0)}{2r} = 1 - \RW^{-N}\,z^{\frac N2}\left( F\bigl(\tfrac{N-2}{2},\tfrac12;\tfrac{N+2}{2};z\bigr) + \frac{z}{N+2}\,F\bigl(\tfrac N2,\tfrac32;\tfrac{N+4}{2};z\bigr)\right) =: 1 - z^ {\frac N2}M(z).
	\end{equation}
	Since the two hypergeometric series have nonnegative coefficients, $M$ is a power series in $z$ with nonnegative coefficients, hence nondecreasing on $[0,1]$; in particular
	\begin{equation}
		\label{eq:minequality}
		M(z) \leq M(1), \qquad\text{for all}\ z\in[0,1].
	\end{equation}
	The value $F(a,b;c;1) = \frac{\Gamma(c)\Gamma(c-a-b)}{\Gamma(c-a)\Gamma(c-b)}$ (see \cite[(9.122.1)]{GR15}) gives
	$$
	F\bigl(\tfrac{N-2}{2},\tfrac12;\tfrac{N+2}{2};1\bigr) = \frac{\sqrt\pi\,\Gamma(\frac N2+1)}{2\,\Gamma(\frac{N+1}{2})}, \qquad
	F\bigl(\tfrac N2,\tfrac32;\tfrac{N+4}{2};1\bigr) = \frac{\sqrt\pi\,\Gamma(\frac N2+2)}{\Gamma(\frac{N+1}{2})},
	$$
	and hence
	\begin{equation}
		\label{eq:valueatone}
		M(1) = \RW^{-N}\left(\frac{\sqrt\pi\,\Gamma(\frac N2+1)}{2\,\Gamma(\frac{N+1}{2})} + \frac{\sqrt\pi\,\Gamma(\frac N2+2)}{(N+2)\,\Gamma(\frac{N+1}{2})}\right) = 1.
	\end{equation}
	Combining \eqref{eq:derivative}, \eqref{eq:minequality} and \eqref{eq:valueatone} yields $\p_r G(r,0)\geq 0$ for $r>\RW$.

    \medskip
    
    \emph{Step 2.} We now prove the assertion of the lemma. Recall $G(r,t) = 2H^{\muW}(x',a+t) + r^2$ with $|x'| = r$. Since $\muW$ is supported on $\hyper_a$, the potential $2H^{\muW}$ is harmonic, hence smooth, in $\{x_N>a\}$; and since $H^{\muW}$ is radial in $x'$, $\p_r(2H^{\muW})$ vanishes at $r=0$ and
	$$
	Q := r^{-1}\p_r G = 2 + r^{-1}\p_r(2H^{\muW})
	$$
	extends to a smooth function of $r$ on $[0,\infty)\times(0,\infty)$.

	In the variables $(r,t)$, $2H^{\muW}$ being harmonic reads $\bigl(\p_t^2 + \p_r^2 + \tfrac{N-2}{r}\p_r\bigr)(2H^{\muW}) = 0$. Applying the same operator to $G$ and using $\Delta_{x'}|x'|^2 = 2(N-1)$,
	$$
	\p_t^2 G + \p_r^2 G + \frac{N-2}{r}\,\p_r G = 2(N-1).
	$$
	Differentiating in $r$ and substituting $\p_rG = rQ$ we obtain
	$$
	\p_t^2 Q + \p_r^2 Q + \frac{N}{r}\,\p_r Q = 0
	\qquad\text{in}\ (r,t)\in(0,\infty)\times(0,\infty) .
	$$
	Since $\p_r^2 + \tfrac{N}{r}\p_r$ is the radial part of the Laplacian on $\R^{N+1}$, the function $\tilde Q$ on $\R^{N+1}\times\R_+$ defined by $\tilde Q(z,t):= Q(|z|,t)$ is harmonic. Moreover, since the density of $\nuW$ is H\"older continuous, the tangential gradient $\nabla_{x'}(2H^{\muW})$ extends continuously up to $\hyper_a$, so $\tilde Q$ is continuous up to $\R^{N+1}\times\{0\}$, where $\tilde Q \geq 0$ by Step 1.

	Furthermore, $\tilde Q(z,t) \to 2$ as $|z|+t\to\infty$, since the second term $r^{-1}\p_r(2H^{\muW})$ in $Q$ tends to $0$. Indeed, write $\hat x = (x', a+t)$, so that $|\hat x|\to\infty$ in $\{x_N> a\}$ along this limit. Since $\muW$ has compact support and bounded density, interior derivative estimates for harmonic functions give
	$$
	|\nabla(2H^{\muW})(\hat x)| \lesssim |\hat x|^{-(N-1)}, \qquad |D^2(2H^{\muW})(\hat x)| \lesssim |\hat x|^{-N}
	\qquad\text{as } |\hat x|\to\infty.
	$$
	For $r \geq 1$, the gradient bound gives 
    $|r^{-1}\p_r(2H^{\muW})| \lesssim |\hat x|^{-(N-1)}$, while for $r < 1$, since $\p_r(2H^{\muW})$ vanishes at $r=0$ by radial symmetry, the mean value theorem gives 
    $$|r^{-1}\p_r(2H^{\muW})| \leq \sup_{[0,r]} |\p_r^2(2H^{\muW})| \lesssim |\hat x|^{-N}.$$
    
    In both cases the second term tends to $0$. Thus, the maximum principle implies that $\tilde Q \geq 0$ on $\R^{N+1}\times\R_+$, as we wanted to prove.
    \end{proof}

    We will also need the following auxiliary computation concerning the function
    \[
    A(t):=-\partial_t H^{\muW}(0,a+t),\qquad t>0.
    \]

    \begin{lemma}\label{l.normal-field}
    The function \(A\) is strictly decreasing on \((0,\infty)\), and
    \[
    \lim_{t\downarrow0}A(t)=a_{\rm c},
    \qquad
    \lim_{t\to\infty}A(t)=0.
    \]
    \end{lemma}
    
    \begin{proof}
    Since
    \[
    d\nuW(y')
    =
    \frac{2\RW\,\Gamma\!\left(\frac{N+1}{2}\right)}{\pi^{\frac{N+1}{2}}}
    \sqrt{\left(1-\frac{|y'|^2}{\RW^2}\right)_+}\,dy',
    \]
    we have, for all \(N\ge2\),
    \[
    A(t)
    =
    t\int_{\mathbb R^{N-1}}\frac{d\nuW(y')}{(|y'|^2+t^2)^{N/2}}.
    \]
    Passing to polar coordinates in \(\mathbb R^{N-1}\), using
    \[
    |\mathbb S^{N-2}|
    =
    \frac{2\pi^{\frac{N-1}{2}}}{\Gamma(\frac{N-1}{2})},
    \qquad
    \Gamma\!\left(\frac{N+1}{2}\right)
    =
    \frac{N-1}{2}\Gamma\!\left(\frac{N-1}{2}\right),
    \]
    gives
    \[
    A(t)
    =
    \frac{2(N-1)\RW}{\pi}\,
    u\int_0^1
    \frac{s^{N-2}(1-s^2)^{1/2}}{(s^2+u^2)^{N/2}}\,ds,
    \qquad u:=\frac{t}{\RW}.
    \]
    Set
    \[
    B(u):=
    u\int_0^1
    \frac{s^{N-2}(1-s^2)^{1/2}}{(s^2+u^2)^{N/2}}\,ds.
    \]
    We show \(B'(u)<0\). Differentiating under the integral sign,
    \[
    B'(u)
    =
    \int_0^1
    \frac{s^{N-2}(1-s^2)^{1/2}\bigl(s^2-(N-1)u^2\bigr)}
         {(s^2+u^2)^{\frac N2+1}}\,ds .
    \]
    Now observe that
    \[
    \frac{d}{ds}
    \left[
    s^{N-1}(1-s^2)^{1/2}(s^2+u^2)^{-N/2}
    \right]
    =
    \frac{s^{N-2}\bigl((N-1)u^2-(Nu^2+1)s^2\bigr)}
         {(1-s^2)^{1/2}(s^2+u^2)^{\frac N2+1}} .
    \]
    The boundary terms at \(s=0\) and \(s=1\) vanish. Hence the integral of
    the right side is zero. Rewriting the numerator in \(B'(u)\) as
    \[
    (1-s^2)\bigl(s^2-(N-1)u^2\bigr)
    =
    -\bigl((N-1)u^2-(Nu^2+1)s^2\bigr)-s^2(s^2+u^2),
    \]
    we obtain
    \[
    B'(u)
    =
    -\int_0^1
    \frac{s^N}{(1-s^2)^{1/2}(s^2+u^2)^{N/2}}\,ds
    <0.
    \]
    Thus \(A\) is strictly decreasing.
    
    It remains to compute the endpoint values. As \(u\downarrow0\),
    \[
    u\int_0^1
    \frac{s^{N-2}(1-s^2)^{1/2}}{(s^2+u^2)^{N/2}}\,ds
    \longrightarrow
    \int_0^\infty \frac{\rho^{N-2}}{(1+\rho^2)^{N/2}}\,d\rho
    =
    \frac{\sqrt\pi\,\Gamma(\frac{N-1}{2})}{2\Gamma(\frac N2)},
    \]
    where the last equality comes from a Beta function identity. Therefore
    \[
    \lim_{t\downarrow0}A(t)
    =
    \frac{(N-1)\RW}{\sqrt\pi}
    \frac{\Gamma(\frac{N-1}{2})}{\Gamma(\frac N2)}.
    \]
    Using
    \[
    \RW^N
    =
    \frac{\sqrt\pi\,\Gamma(\frac N2+1)}{\Gamma(\frac{N+1}{2})}
    =
    \frac{N\sqrt\pi\,\Gamma(\frac N2)}
    {(N-1)\Gamma(\frac{N-1}{2})},
    \]
    this becomes
    \[
    \lim_{t\downarrow0}A(t)=\frac{N}{\RW^{N-1}}=a_{\rm c}.
    \]
    Finally, \(A(t)\to0\) as \(t\to\infty\), for instance by dominated
    convergence in the original integral.
    \end{proof}

    With these lemmas in hand, we can finally prove Proposition \ref{p.threshold}.
    \begin{proof}[Proof of Proposition \ref{p.threshold}]
        By Lemma \ref{l.char} (whose proof is based on the strict convexity of $\mu\mapsto\energy[\mu]$), it suffices to show that there is a constant $c\in\R$ such that
    \begin{equation}
    	\label{eq:elgoal}
    	\begin{cases}
    		2H^{\muW}(x) + |x|^2 = c & \text{on}\ \{ |x'|\leq \RW\} \times \{ x_N = a\} \,, \\
    		2H^{\muW}(x) + |x|^2 \geq c & \text{on}\ \{x_N\geq a\} \,.
    	\end{cases}
    \end{equation}
    
    We will show \eqref{eq:elgoal} with $c=c' + a^2$ with $c'$ as in Lemma \ref{l.hyperplane}. On $\{x_N = a\}$, \eqref{eq:elgoal} follows directly from \eqref{eq:eln-1}.
    
    Now we focus on $\{x_N > a\}$. Since $H^{\muW}$ depends on $x'$ only through $r = |x'|$, we may write
    $$\Phi(r,t) := 2H^{\muW}(x) + |x|^2,$$
    with $x = (x',a+t)$ and $r = |x'|$.

    By Lemma \ref{l.normal-field} and using $a \geq a_c$ and $t\geq 0$,
    $$\p_t\Phi(0,t) = -2A(t)+2(a+t) \geq -2A(t) + 2a_c \geq 0,$$
    and then $\Phi(0,t) \geq \Phi(0,0) = c$ for all $t \geq 0$.
    
    Moreover,
    $$\Phi(r,t) = G(r,t) + (a+t)^2,$$
    where $G$ is as in Lemma \ref{mono}, and then $\partial_r \Phi(r,t) = \partial_r G(r,t) \geq 0$. Hence,
    $$\Phi(r,a+t) \geq \Phi(0,a+t) \geq \Phi(0,0) = c.$$
    This completes the proof of the proposition.
    \end{proof}

\section{Classification}\label{s.classification}

    In this section we prove Theorem \ref{t.classification}. First, we see that for all $a > 0$, there must be some mass supported in $\hyper_a$.
	\begin{proposition}\label{p.singular}
		For every $N\ge2$ and for every $a>0$, if $\lambda(a)>0$ then
		\begin{equation*}
			a\ \le\ N\lambda(a)^{-(N-1)/N}\bigl(1-\lambda(a)\bigr).
		\end{equation*}
		In particular $\lambda(a)<1$, which together with Theorem~\ref{t.structure}(b) gives $\muhat_{S,\kern0.2pt a}\ne0$.
	\end{proposition}
    
	\begin{proof}
        For all $x\in\bulk_a$, since $\bulk_a$ is open, we can differentiate Theorem \ref{t.structure}(a) to obtain
        \begin{equation}\label{e.fieldid}
            x = -\nabla\tfrac12(c-|x|^2) = -\nabla H^{\muhat_a} = -\nabla\gN * \muhat_a = \int \frac{x-y}{|x-y|^N}\,d\muhat_a(y).
        \end{equation}
    
		Now assume $\lambda:=\lambda(a)>0$. Integrate \eqref{e.fieldid} over $\bulk_a$ against the Lebesgue
		measure and divide by~$|B_1|$:
		\[
		\frac{1}{|B_1|}\int_{\bulk_a}\!\Bigl(\int\frac{x-y}{|x-y|^N}\,d\muhat_a(y)\Bigr)dx
		=\frac{1}{|B_1|}\int_{\bulk_a}x\,dx=\lambda\bar x,
		\qquad
		\bar x:=\frac{1}{|\bulk_a|}\int_{\bulk_a}x\,dx .
		\]
		Split $d\muhat_a=\tfrac1{|B_1|}\chi_{\bulk_a}\,dx+d\nu_a$ in the inner integral, and note that $\nu_a:=\muhat_a\!\restriction_{\R^N\setminus\bulk_a}$ has total mass $1-\lambda$. Now, by Fubini, and using that the bulk-bulk interaction term is antisymmetric,
        \begin{align*}
            \lambda\bar x &= \frac{1}{|B_1|}\int_{\bulk_a}\!\Bigl(\int_{\bulk_a}\frac{x-y}{|x-y|^N}\,d\muhat_a(y)\Bigr)dx + \frac{1}{|B_1|}\int_{\bulk_a}\!\Bigl(\int\frac{x-y}{|x-y|^N}\,d\nu_a(y)\Bigr)dx\\
            &= \frac{1}{|B_1|^2}\int_{\bulk_a}\int_{\bulk_a}\frac{x-y}{|x-y|^N}\,dxdy + \int\Bigl(\frac{1}{|B_1|}\int_{\bulk_a}\frac{x-y}{|x-y|^N}\,dx\Bigr)d\nu_a(y)\\
            &= \int\Bigl(\frac{1}{|B_1|}\int_{\bulk_a}\frac{x-y}{|x-y|^N}\,dx\Bigr)d\nu_a(y).
        \end{align*}

        Therefore, since $|\bulk_a| = \lambda|B_1| = |B_{\lambda^{1/N}}|$, by the bathtub principle \cite[Theorem 1.14]{LL01},
        $$\lambda|\bar x| \leq \frac{1}{|B_1|}\int\Bigl(\int_{\bulk_a}\frac{dx}{|x-y|^{N-1}}\Bigr)d\nu_a(y) \leq \frac{1}{|B_1|}\int_{B_{\lambda^{1/N}}}\frac{dx}{|x|^{N-1}}|\nu_a| = N\lambda^{1/N}(1-\lambda).$$
        On the other hand, since $\muhat_a$ is supported in $\half_a$, $|\bar x| \geq a$, and then the conclusion follows.
	\end{proof}

    \begin{remark}
        The constant $N$ in the proposition can be improved. Indeed, to bound
        $$
        \int_{\bulk_a}\frac{(\bar x/|\bar x|)\cdot(x-y)}{|x-y|^{N}}\,dx
        $$
        from above, in the previous proof we used $(\bar x/|\bar x|)\cdot(x-y) \leq |x-y|$ and then we applied a rearrangement inequality. Instead we can use the explicit solution of the minimization problem
        $$
        \sup_{|E|=V} \int_E \frac{z\cdot e}{|z|^N}\,dz
        $$
        where $V>0$ and $e\in\Sph^{N-1}$ are given. For a similar computation, see \cite[Lemma 11.5]{FH25}.
    \end{remark}
    
    Finally, we prove our classification theorem.
	\begin{proof}[Proof of Theorem~\ref{t.classification}]
		First, if $a \leq -1$, then $\mustar \in \Cset_a$, and since $\mustar$ is the unconstrained global minimizer in $\Prob(\R^N)$, it is also the global minimizer in the smaller set $\Cset_a$.
		
		Then, if $a \in (-1,\ac)$, $\muhat_a$ is not purely singular. This is already shown in \cite{BFMS26}, but it can also be derived independently from what we have already shown. Indeed, suppose $\muhat_a = \muhat_{S,\kern0.2pt a}$. Then $\muhat_a$ is supported on $\hyper_a$, and then by Lemma \ref{l.hyperplane} $\muhat_a = \muW$. 
        
        Writing
        $\Phi(r,t) := 2H^{\muW}(x) + |x|^2$ with $x = (x',a+t)$, $r = |x'|$, as in the proof of Proposition \ref{p.threshold}, Lemma \ref{l.char} gives $\Phi \geq c_a$ on $\half_a$ with equality on
        $\supp\muW$. As $(0,a)\in\supp\muW$, $\Phi$ attains its minimum at $(0,0)$, but Lemma \ref{l.normal-field} gives
        $$\p_t\Phi(0,0^+) = -2A(0^+) + 2a = -2a_c + 2a < 0,$$
        a contradiction.

        Now we distinguish two cases: if $a \in (-1,0]$, then $\muhat_{S,\kern0.2pt a} \neq 0$ by Theorem \ref{t.structure}(d), and if $a \in (0,\ac)$, then $\muhat_{S,\kern0.2pt a} \neq 0$ by Proposition~\ref{p.singular}.

        Finally, if $a \geq \ac$, then $d\muhat_a = d\nuW(x')\otimes\delta_a(x_N)$ by \cite{DKMSS17} in one dimension, by \cite[Proposition 3.1]{ASZ14} in two dimensions, and by Proposition \ref{p.threshold} for $N \geq 3$.
	\end{proof}

\end{document}